\documentclass[12pt,reqno]{amsart}
\usepackage{amscd,amsfonts,amsmath,amssymb,amsthm}
\usepackage{dsfont,enumerate,graphicx,hyperref,pgfplots,perpage,url}

\usepackage[margin=2.2cm]{geometry}

\theoremstyle{plain}
\newtheorem{thm}{Theorem}[section]

\newtheorem{cor}[thm]{Corollary}
\newtheorem{defn}[thm]{Definition}

\newtheorem{prop}[thm]{Proposition}

\theoremstyle{remark}

\newtheorem*{rmk}{\textbf{Remark}}

\numberwithin{equation}{section}

\MakePerPage{footnote} \allowdisplaybreaks 

\newcommand\Ab{\mathbb A}
\newcommand\Ac{\mathcal A}
\newcommand\C{\mathbb{C}}
\newcommand\Cc{\mathcal{C}}
\newcommand\Card{\mathrm{Card}}
\newcommand\Char{\mathds 1}

\newcommand\diam{\mathrm{diam}}

\newcommand\Eb{\mathbb E}
\newcommand\Fc{\mathcal F}
\newcommand\Gb{\mathbb G}

\newcommand\Lip{\mathrm{Lip}}
\newcommand\Mcc{\mathcal{M}}
\newcommand\N{\mathbb{N}}
\newcommand\ol{\overline}

\newcommand\R{\mathbb R}

\newcommand\Sp{\mathrm{Sp}}

\newcommand\T{\mathbb{T}}
\newcommand\ve{\varepsilon}
\newcommand\vk{\varkappa}

\newcommand\Xc{\mathcal{X}}
\newcommand\Z{\mathbb{Z}}

\title{Fractal uncertainty principle for discrete Cantor sets with random alphabets}

\author{Suresh Eswarathasan}
\email{sr766936@dal.ca}
\address{Department of Mathematics and Statistics, Dalhousie University, Halifax, NS B3H1Z9, Canada}

\author{Xiaolong Han}
\email{xiaolong.han@csun.edu}
\address{Department of Mathematics, California State University, Northridge, CA 91330, USA}



\thanks{} 

\begin{document}
\maketitle

\begin{abstract}
In this paper, we investigate the fractal uncertainty principle (FUP) for discrete Cantor sets, which are determined by an alphabet from a base of digits. Consider the base of $M$ digits and the alphabets of cardinality $A$ such that all the corresponding Cantor sets have a fixed dimension $\log A/\log M\in(0,2/3)$. We prove that the FUP with an improved exponent over Dyatlov-Jin \cite{DJ1} holds for almost all alphabets, asymptotically as $M\to\infty$. Our result provides the best possible exponent when the Cantor sets enjoy either the strongest Fourier decay assumption or strongest additive energy assumption. The proof is based on a concentration of measure phenomenon in the space of alphabets.
\end{abstract}

\section{Introduction}\label{sec:intro}
The fractal uncertainty principle (FUP) was recently introduced by Dyatlov-Zahl \cite{DyZa} and has quickly become an emerging topic in Fourier analysis. It concerns the phenomenon that no function can be localized in both position and frequency close to a fractal set. More precisely, the FUP is formulated in the context of estimating the norm
\begin{equation}\label{eq:FUP}
\|\Char_X\Fc_h\Char_Y\|_{L^2(\R^d)\to L^2(\R^d)},
\end{equation}
in which $0<h<1$ is the semiclassical parameter, the $h$-dependent sets $X=X(h),Y=Y(h)\subset\R^d$ are equipped with certain fractal-type structures, and $\Fc_h$ is the semiclassical Fourier transform
$$\Fc_hf(\xi)=\frac{1}{(2\pi h)^\frac d2}\int_{\R^d}e^{-\frac{ix\cdot \xi}{h}}f(x)\,dx.$$
Since $\Fc_h$ is unitary in $L^2(\R^d)$, we always have that $\eqref{eq:FUP}\le1$. Following Dyatlov \cite[Definition 2.1]{Dy}, we say that $X,Y$ satisfy the FUP with exponent $\beta\ge0$ if $\eqref{eq:FUP}=O(h^\beta)$ as $h\to0$. It can be understood as one quantitative version of the uncertainty principle that no function can be localized both near $Y$ in the position space and near $X$ in the frequency space. The main objective about the FUP is to prove the existence of exponent $\beta>0$ and to find the sharp one for which the FUP holds. Here, ``the sharp exponent'', denoted by $\beta^s(X,Y)$, means the largest exponent such that $\eqref{eq:FUP}=O(h^\beta)$ holds for all $0<h<h_0$ with some $h_0>0$. It usually depends on the regularity and dimensions of $X$ and $Y$.

In this paper, via the discrete Fourier transform, we establish a new connection: the FUP and the concentration of measure phenomenon in asymptotic geometric analysis and probability \cite{MS, L}. We follow Dyatlov-Jin \cite{DJ1} and briefly recall the setup of the FUP for discrete Cantor sets. Each discrete Cantor set $\Cc_k(M,\Ac)$ is determined by a \textit{base} $\{0,...,M-1\}$ with $M\in\N$ and $M\ge3$, an \textit{alphabet} $\Ac\subset\{0,...,M-1\}$, and an \textit{order} $k\in\N$:
\begin{equation}\label{eq:Cantor}
\Cc_k(M,\Ac)=\left\{\sum_{j=0}^{k-1}a_jM^j:a_0,\dots,a_{k-1}\in\Ac\right\}.
\end{equation}
It is evident that $\Card(\Cc_k(M,\Ac))=A^k$, where $A=\Card(\Ac)$, the cardinality of $\Ac$.

Let $N\in\N$ and $l^2_N$ be the Hilbert space of functions $u:\{0,...,N-1\}\to\C$ with norm 
$$\|u\|_{l^2_N}^2=\sum_{j=0}^{N-1}|u(j)|^2.$$
Define the (unitary) discrete Fourier transform $\Fc:l^2_N\to l^2_N$
\begin{equation}\label{eq:FN}
\Fc_N u(j)=\frac{1}{\sqrt N}\sum_{l=0}^{N-1}e^{-\frac{2\pi ijl}{N}}u(l).
\end{equation}
For the discrete Cantor sets $\Cc_k=\Cc_k(M,\Ac)$ and $N=M^k$ with $k\in\N$, we consider the FUP \eqref{eq:FUP} in terms of the operator norm 
\begin{equation}\label{eq:rk}
r_k:=\|\Char_{\Cc_k}\Fc_N\Char_{\Cc_k}\|_{l^2_N\to l^2_N}.
\end{equation}
Since $\Fc_N$ is unitary, $\|\Char_{\Cc_k}\Fc_N\Char_{\Cc_k}\|_{l^2_N\to l^2_N}\le1$. Throughout the paper, we denote
\begin{equation}\label{eq:delta}
\delta=\frac{\log A}{\log M}.
\end{equation}
Note that $\delta$ is the dimension of the Cantor set
\begin{equation}\label{eq:Cinfty}
\Cc_\infty=\bigcap_{k=0}^\infty\bigcup_{j\in\Cc_k}\left[\frac{j}{M^k},\frac{j+1}{M^k}\right].
\end{equation}
The FUP for discrete Cantor sets \eqref{eq:rk} in two extreme cases can be trivially derived: 
\begin{enumerate}[(i).]
\item If $\delta=0$ (i.e., $A=1$), then $\Card(\Cc_k(M,\Ac))=1$ so $r_k=N^{-1/2}$ for all $k\in\N$.
\item If $\delta=1$ (i.e., $A=M$), then $\Cc_k(M,\Ac)=\{0,...,N-1\}$ and $\Char_{\Cc_k}\Fc_N\Char_{\Cc_k}=\Fc_N$ so $r_k=1$ for all $k\in\N$.
\end{enumerate} 
We therefore consider the cases when $0<\delta<1$ (i.e., $1<A<M$) in the following discussion. In these cases, the mapping $\Char_{\Cc_k}\Fc_N\Char_{\Cc_k}:l^2_N\to l^2_N$ is given by a symmetric $N\times N$ matrix, which contains exactly $A^k$ non-zero rows and columns (of the form $e^{-2\pi ijl/N}/\sqrt N$). Hence, 
\begin{equation}\label{eq:volbd}
r_k=\left\|\Char_{\Cc_k}\Fc_N\Char_{\Cc_k}\right\|_{l^2_N\to l^2_N}\le\left\|\Char_{\Cc_k}\Fc_N\Char_{\Cc_k}\right\|_{\mathrm{HS}}=\sqrt{\frac{\Card(\Cc_k)^2}{N}}=M^{-k\left(\frac12-\delta\right)}=N^{-\left(\frac12-\delta\right)},
\end{equation}
in which $N=M^k\to\infty$ as $k\to\infty$ and $1/N\to0$ plays the role of the semiclassical parameter $h$ in \eqref{eq:FUP}. Here, $\|\cdot\|_{\mathrm{HS}}$ is the Hilbert-Schmidt norm. 

From \eqref{eq:volbd}, an FUP with exponent $\beta=1/2-\delta>0$ follows if $\delta<1/2$ (i.e., $A=M^\delta<M^{1/2}$). Such estimate is usually referred as ``the volume bound'' because it only takes the size of $\Cc_k$ into consideration. One then asks whether these bounds can be improved, that is,
$$r_k\le N^{-\beta}\quad\text{for some }\beta>\max\left\{0,\frac12-\delta\right\}.$$
The first of such results for all $0<\delta<1$ was proved by Dyatlov-Jin \cite[Theorem 2]{DJ1}:
\begin{thm}\label{thm:DJ}
Let $\Ac\subset\{0,...,M-1\}$ with $M\ge3$ and $1<\Card(\Ac)<M$. Then there exists 
$$\beta=\beta(M,\Ac)>\max\left(0,\frac12-\delta\right)\quad\text{such that}\quad r_k\le N^{-\beta}\quad\text{for all }k\in\N.$$
\end{thm}
In particular, we define the sharp exponent 
$$\beta^s(M,\Ac):=\sup\left\{\beta\ge0:r_k\le N^{-\beta}\text{ for all }k\in\N\right\}.$$
Then $\beta^s(M,\Ac)>0$ for all discrete Cantor sets. However, the characterization of $\beta^s(M,\Ac)$ and its dependence on $M$ and $\Ac$ is not yet clear. Firstly,
\begin{rmk}[The best possible exponent in the FUP for all discrete Cantor sets]
Take $u(j)=1$ for some $j\in\Cc_k$ and $u(j)=0$ otherwise. Then 
$$\|u\|_{l^2_N}=1\quad\text{and}\quad\|\Char_{\Cc_k}\Fc_N\Char_{\Cc_k}u\|_{l^2_N}=\sqrt{\frac{\Card(\Cc_k)}{N}}=\sqrt{\frac{A^k}{N}}=N^{-\frac{1-\delta}{2}},$$
in which $\Card(\Cc_k)=A^k=M^{\delta k}=N^\delta$. Hence, $r_k\ge N^{-(1-\delta)/2}$ and the best possible exponent $\beta$ in the FUP is $(1-\delta)/2$, that is, for any $M$ and $\Ac$,
\begin{equation}\label{eq:upper}
\beta^s(M,\Ac)\le\frac{1-\delta}{2}.
\end{equation}
\end{rmk}

Dyatlov-Jin \cite[Section 3.5]{DJ1} provided examples of discrete Cantor sets for which the FUP with the best possible exponent \eqref{eq:upper} holds. On the other hand, they also found examples for which 
$$\beta^s(M,\Ac)\le\frac12-\delta+\frac{1}{KM\log M}\quad\text{and}\quad\left|\delta-\frac12\right|\le\frac{1}{K\log M}.$$
Here, $K>0$ is an absolute constant. This sharp exponent improves over the volume bound \eqref{eq:volbd} by a polynomial term in $M$ and is much smaller than the best possible one \eqref{eq:upper} \cite[Proposition 3.17]{DJ1}. These examples painted a complicated picture of the sharp exponents in the FUP for discrete Cantor sets with various bases and alphabets. See Dyatlov-Jin \cite[Figure 3]{DJ1} for the numerical results of $\beta^s(M,\Ac)$ for all alphabets when $3\le M\le10$. Despite the different estimates of $\beta^s(M,\Ac)$ in the FUP for individual alphabets and bases, Dyatlov-Jin \cite[Section 3.5]{DJ1} observed that for fixed $M$ and $A$, the expected $\beta^s(M,\Ac)$ for all alphabets $\Ac$ with $\Card(\Ac)=A$ appears to be much larger than the volume bound \eqref{eq:volbd}, see the solid blue curve in Figure \ref{fig:Eb}. 

The purpose of this paper is to provide a rigorous explanation of the observation in Dyatlov-Jin \cite{DJ1}. To this end, write the space of alphabets as
$$\Ab(M,A):=\left\{\Ac\subset\{0,...,M-1\}:\Card(\Ac)=A\right\},$$
which has cardinality 
$$\Card(\Ab(M,A))={M\choose A}.$$
Each element $\Ac\in\Ab(M,A)$ is an alphabet with cardinality $A$ and defines a collection of discrete Cantor sets $\Cc_k(M,\Ac)$ by \eqref{eq:Cantor}. Set the uniform counting probability measure $\mu$ on $\Ab(M,A)$, that is,
\begin{equation}\label{eq:prob}
\mu(\Omega)=\frac{\Card(\Omega)}{\Card(\Ab(M,A))}\quad\text{for any }\Omega\subset\Ab(M,A).
\end{equation}
Our main theorem states that
\begin{thm}\label{thm:FUPC}
Let $M,A\in\N$ with $M\ge3$ and $\delta=\log A/\log M\in(0,2/3)$. Suppose that $\ve>0$. Then there exists $\Gb=\Gb(M,\ve)\subset\Ab(M,A)$ with
$$\mu\big(\Ab(M,A)\setminus\Gb\big)\le4Me^{-\frac{M^{4\ve}}{64}}$$
such that for all $\Ac\in\Gb$,
$$\beta^s(M,\Ac)\ge\frac12-\frac34\delta-\ve.$$
\end{thm}

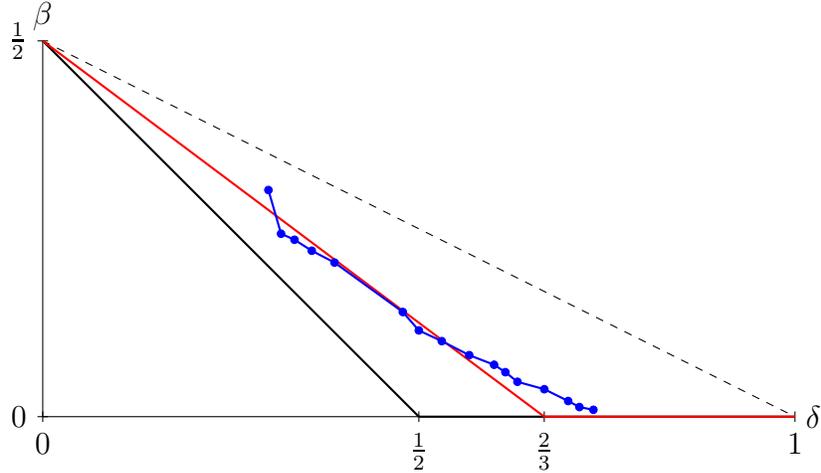
\begin{figure}[h!]
\begin{tikzpicture}[]
\draw[-] (0,0) -- (10,0) node[right] {$\delta$};
\draw[-] (0,0) -- (0,5) node[above] {$\beta$};

\foreach \x/\xtext in {0/0, 5/{\frac12}, {20/3}/{\frac23}, 10/1}
\draw[shift={(\x,0)}] (0pt,2pt) -- (0pt,-2pt) node[below] {$\xtext$};

\foreach \y/\ytext in {0/0, 5/\frac12}
\draw[shift={(0,\y)}] (2pt,0pt) -- (-2pt,0pt) node[left] {$\ytext$};

\draw[dashed] (0,5) -- (10,0);
\draw[thick] (0,5) -- (5,0);
\draw[thick] (5,0) -- (10,0);
\draw[red,thick] (0,5) -- (20/3,0) -- (10,0);

\draw[blue] (3,8.25*5/13.68) circle (0.05cm);
\fill[blue] (3,8.25*5/13.68) circle (0.05cm);

\draw[blue] (86.9/27.43,6.66*5/13.68) circle (0.05cm);
\fill[blue] (86.9/27.43,6.66*5/13.68) circle (0.05cm);

\draw[blue] (91.8/27.43,6.44*5/13.68) circle (0.05cm);
\fill[blue] (91.8/27.43,6.44*5/13.68) circle (0.05cm);

\draw[blue] (98.1/27.43,6.04*5/13.68) circle (0.05cm);
\fill[blue] (98.1/27.43,6.04*5/13.68) circle (0.05cm);

\draw[blue] (106.4/27.43,5.61*5/13.68) circle (0.05cm);
\fill[blue] (106.4/27.43,5.61*5/13.68) circle (0.05cm);

\draw[blue] (131.3/27.43,3.81*5/13.68) circle (0.05cm);
\fill[blue] (131.3/27.43,3.81*5/13.68) circle (0.05cm);

\draw[blue] (5,3.14*5/13.68) circle (0.05cm);
\fill[blue] (5,3.14*5/13.68) circle (0.05cm);

\draw[blue] (145.5/27.43,2.75*5/13.68) circle (0.05cm);
\fill[blue] (145.5/27.43,2.75*5/13.68) circle (0.05cm);

\draw[blue] (155.5/27.43,2.24*5/13.68) circle (0.05cm);
\fill[blue] (155.5/27.43,2.24*5/13.68) circle (0.05cm);

\draw[blue] (6,1.89*5/13.68) circle (0.05cm);
\fill[blue] (6,1.89*5/13.68) circle (0.05cm);

\draw[blue] (168.7/27.43,1.62*5/13.68) circle (0.05cm);
\fill[blue] (168.7/27.43,1.62*5/13.68) circle (0.05cm);

\draw[blue] (173.1/27.43,1.27*5/13.68) circle (0.05cm);
\fill[blue] (173.1/27.43,1.27*5/13.68) circle (0.05cm);

\draw[blue] (182.9/27.43,1*5/13.68) circle (0.05cm);
\fill[blue] (182.9/27.43,1*5/13.68) circle (0.05cm);

\draw[blue] (191.6/27.43,0.57*5/13.68) circle (0.05cm);
\fill[blue] (191.6/27.43,0.57*5/13.68) circle (0.05cm);

\draw[blue] (195.7/27.43,0.35*5/13.68) circle (0.05cm);
\fill[blue] (195.7/27.43,0.35*5/13.68) circle (0.05cm);

\draw[blue] (200.8/27.43,0.25*5/13.68) circle (0.05cm);
\fill[blue] (200.8/27.43,0.25*5/13.68) circle (0.05cm);

\draw[blue,thick] (3,8.25*5/13.68) -- (86.9/27.43,6.66*5/13.68) -- (91.8/27.43,6.44*5/13.68) -- (98.1/27.43,6.04*5/13.68) -- (106.4/27.43,5.61*5/13.68) -- (131.3/27.43,3.81*5/13.68) -- (5,3.14*5/13.68) -- (145.5/27.43,2.75*5/13.68) -- (155.5/27.43,2.24*5/13.68) -- (6,1.89*5/13.68) -- (168.7/27.43,1.62*5/13.68) -- (173.1/27.43,1.27*5/13.68) -- (182.9/27.43,1*5/13.68) -- (191.6/27.43,0.57*5/13.68) -- (195.7/27.43,0.35*5/13.68) -- (200.8/27.43,0.25*5/13.68);
\end{tikzpicture}
\caption{\small{The solid black line is the volume bound $\frac12-\delta$ \eqref{eq:volbd}; the dashed black line is the best possible exponent $\frac{1-\delta}{2}$ \eqref{eq:upper}. The solid red line is $\frac12-\frac34\delta$ in Theorem \ref{thm:Eb}; the solid blue curve is the numerical results of the average of $\beta^s(M,\Ac)$ over all alphabets with given $3\le M\le10$ and $\Card(\Ac)$ (reproduced from Dyatlov-Jin \cite[Figure 3]{DJ1} with permission).}}\label{fig:Eb}
\end{figure}


That is, outside of a set of alphabets with exponentially small probability, 
$$\beta^s(M,\Ac)\ge\frac12-\frac34\delta-\ve$$ 
for discrete Cantor sets with random alphabets $\Ac$ such that $\Card(\Ac)=M^\delta$. Together with the trivial bound\footnote{Here, one could use the volume bound that $\beta^s(M,\Ac)\ge1/2-\delta$ for $\delta\in(0,1/2)$, or the improved bound for $\delta\in(0,1)$ in Theorem \ref{thm:DJ}. It would lead to better estimates of $\Eb(\beta^s(M,\Ac))$ than the one in Theorem \ref{thm:Eb}. For small $M$, such discrepancy (and numerical error) may be responsible for the fluctuation of the numerical results in the blue curve, comparing with the red line, in Figure \ref{fig:Eb}. However, as $M\to\infty$, the discrepancy disappears and the conclusion in Theorem \ref{thm:Eb} stays the same.} that $\beta^s(M,\Ac)\ge0$, we have that for any $\tilde\ve>0$,
$$\Eb(\beta^s(M,\Ac))\ge\left(1-4Me^{-\frac{M^{4\ve}}{64}}\right)\left(\frac12-\frac34\delta-\ve\right)\ge\frac12-\frac34\delta-\tilde\ve,$$ 
if $M$ is sufficiently large. Thus,
\begin{thm}\label{thm:Eb}
Let $M,A\in\N$ with $M\ge3$ and $1<A<M$. Then the expectation of $\beta^s(M,\Ac)$ satisfies that
$$\Eb(\beta^s(M,\Ac))\ge\max\left\{0,\frac12-\frac34\delta+o_M(1)\right\}.$$
\end{thm}

The FUP has been addressed by various authors including Dyatlov-Zahl \cite{DyZa}, Dyatlov-Jin \cite{DJ1, DJ3}, Bourgain-Dyatlov \cite{BD1, BD2}, Jin-Zhang \cite{JZ}, Han-Schlag \cite{HS}, Dyatlov-Zworski \cite{DyZw}, and Cladek-Tao \cite{CT}, etc. Among these works, the arguments of Fourier decay from harmonic analysis and additive energy from combinatorics were introduced to establish the FUP \cite{BD1, CT, DJ1, DyZa}; the best possible exponent in the FUP that one can achieve using these arguments is $\frac12-\frac34\delta$ (i.e., when the fractal sets in question enjoy the strongest Fourier decay\footnote{We thank Long Jin for pointing this out to us.} or the strongest additive energy bound). See also Dyatlov \cite[Sections 5.1 and 5.2]{Dy} for an overview. By means of comparison, Theorems \ref{thm:FUPC} and \ref{thm:Eb} imply that the FUP with such exponent holds for almost all alphabets, despite the fact that the strongest Fourier decay and the strongest additive energy bound are both unknown for the corresponding Cantor sets.

The FUP has also found a wide range of applications including the scattering theory in open chaotic systems \cite{DyZa, DJ1, DJ3, BD1, BD2}, the spectral theory in closed chaotic systems (i.e., quantum chaos) \cite{DJ2, DJN, Schw}, partial different equations \cite{J1, J2, W, GZ}, etc. See Dyatlov \cite{Dy} and Dang \cite{Da} for recent surveys on the topic of the FUP. 

We now derive an immediate consequence of the FUP in Theorems \ref{thm:FUPC} and \ref{thm:Eb} for open quantum maps (also called the ``quantized open baker's map''), a popular model for the study of scattering theory in open chaotic systems. Again we follow Dyatlov-Jin's program and refer to their paper \cite{DJ1} for a detailed presentation of the background and related literature.  

For $M,A\in\N$ with $M\ge3$ and $1<A<M$, an open quantum map $B_N$ is defined as
$$B_N=B_{N,\chi}=\Fc^{-1}_N\begin{pmatrix}
\chi_{N/M}\Fc_{N/M}\chi_{N/M} & & \\
 & \ddots & \\
  & & \chi_{N/M}\Fc_{N/M}\chi_{N/M}
\end{pmatrix}.$$
Here, $\Fc_N$ and $\Fc_N^{-1}$ is the discrete Fourier transform \eqref{eq:FN} and its inverse, $N=M^k\to\infty$ as $k\to\infty$, and $\chi$ is a cutoff function.  For intuition, we consider the following elementary example.  Let $M=3$ and $\Ac=\{0,2\}$. Then
$$B_N=\Fc^{-1}_N\begin{pmatrix}
\chi_{N/3}\Fc_{N/3}\chi_{N/3} & 0 & 0\\
0 & 0 & 0\\
0 & 0 & \chi_{N/3}\Fc_{N/3}\chi_{N/3}
\end{pmatrix}.$$
The operator $B_N$ is a quantization of the (classical) open baker's map $\vk_{M,\Ac}:\T^2\to\T^2$ defined as
$$\vk_{M,\Ac}(x,\xi)=\left(Mx-a,\frac{\xi+1}{M}\right)\quad\text{if }(x,\xi)\in\left(\frac aM,\frac{a+1}{M}\right)\times(0,1)\text{ for }a\in\Ac.$$
Formally, $\vk_{M,\Ac}$ sends the points which are not in $(a,a+1)/M\times(0,1)$, $a\in\Ac$, to infinity, and therefore is ``open''. The iterations $\vk_{M,\Ac}^k$, $k\in\Z$, induce a discrete dynamical system in the phase space $\T^2$. This dynamical system is chaotic with trapped set exactly as the Cantor set $\Cc_\infty$ \eqref{eq:Cinfty}, i.e., the points in this set stay in the finite region under the map $\vk_{M,\Ac}^k$ as $k\to\infty$ and $k\to-\infty$.  

The scattering resonances of the open quantum system correspond to the eigenvalues of the matrix $B_N$. The spectrum of $B_N$, denoted by $\Sp(B_N)$, is affected by the fractal structure of the trapped set, $\Cc_\infty$, of the classical dynamics of $\vk_{M,\Ac}$. In particular, the matrix $B_N$ has norm bounded by $1$ and the spectral gap of $B_N$, i.e., the distance between $\Sp(B_N)$ and $1$, is directly provided by the operator norm in \eqref{eq:rk}. (See Dyatlov-Jin \cite[Section 5]{DJ1} for the independence of $\Sp(B_N)$ on the cutoff function $\chi$.) As a consequence of the FUP in Theorem \ref{thm:DJ}, Dyatlov-Jin \cite[Theorem 1]{DJ1} proved that 
\begin{thm}[Spectral gaps of open quantum maps]\label{thm:gapDJ}
Let $\delta\in(0,1)$. Then there exists 
$$\beta=\beta(M,\Ac)>\max\left\{0,\frac12-\delta\right\}$$
such that $B_N$ has a ``spectral gap'' of $\beta$, that is,
$$\limsup_{N\to\infty}\max\{|\lambda|:\lambda\in\Sp(B_N)\}\le M^\beta.$$
\end{thm}
Following the same line, we have that the following corollary as a consequence of the FUP in Theorems \ref{thm:FUPC} and \ref{thm:Eb}.
\begin{cor}[Spectral gaps of open quantum maps associated with random alphabets]\label{cor:gap}\hfill
\begin{enumerate}[(i).]
\item Let $\delta\in(0,2/3)$ and $\ve>0$. Then there exists $\Gb=\Gb(M,\ve)\subset\Ab(M,A)$ with
$$\mu\big(\Ab(M,A)\setminus\Gb\big)\le4Me^{-\frac{M^{4\ve}}{64}}$$
such that for all $\Ac\in\Gb$, the corresponding open quantum map $B_N$ has a spectral gap of at least 
$$\frac12-\frac34\delta-\ve.$$
\item Let $\delta\in(0,1)$. Then the expectation of the spectral map of the corresponding open quantum maps is at least
$$\max\left\{0,\frac12-\frac34\delta+o_M(1)\right\}.$$
\end{enumerate}
\end{cor}

\subsection*{Outline of the proof}
We briefly outline the proof of the FUP in Theorem \ref{thm:FUPC} and the organization of the paper. For discrete Cantor sets $\Cc_k$, the FUP concerns the operator norm $r_k=\|\Char_{\Cc_k}\Fc_N\Char_{\Cc_k}\|_{l^2_N\to l^2_N}$ as $k\to\infty$ in \eqref{eq:FUP}. Such estimate is greatly simplified due to the algebraic structure of $\Cc_k$. That is, Dyatlov-Jin \cite[Section 3.1]{DJ1} proved a ``submultiplicativity property'' that $r_{k_1+k_2}\le r_{k_1}r_{k_2}$ for all $k_1,k_2\in\N$ (see Proposition \ref{prop:subm}). Therefore, one reduces the FUP to the estimate of $r_1=\|\Char_{\Ac}\Fc_M\Char_{\Ac}\|_{l^2_M\to l^2_M}$, which is the singular value of the corresponding matrix of $\Char_{\Ac}\Fc_M\Char_{\Ac}$, denoted by $\Mcc(\Ac)$. 

The matrix $\Mcc(\Ac)$ has exactly $A=\Card(\Ac)$ non-zero rows and columns, and non-zero entries of the form $e^{-2\pi ijl/M}/\sqrt M$ for $j,l\in\Ac$, in the view of \eqref{eq:FN}. We estimate $r_1$ via the analysis of $\Mcc(\Ac)\Mcc(\Ac)^\star$, which has non-zero entries of the form
$$\frac1M\sum_{l\in\Ac}e^{\frac{2\pi(k-j)l}{M}}\quad\text{for }k,j\in\Ac.$$
By Schur's lemma (see Dyatlov-Jin \cite[Lemma 3.8]{DJ1} and Dyatlov \cite[Equation (4.11)]{Dy}), we have that
$$r_1^2\le\sup_{j\in\Ac}\sum_{k\in\Ac}\left(\frac1M\sum_{l\in\Ac}e^{\frac{2\pi(k-j)l}{M}}\right).$$
 We are then led to study the exponential sum $\sum_{l\in\Ac}e^{2\pi(k-j)l/M}$ for $k,j\in\Ac$ with $k\ne j$ and its dependence on the random alphabets $\Ac\in\Ab(M,A)$. 

To this end, we build a concentration of measure theory in the space $\Ab(M,A)$, see Section \ref{sec:com}. In short, the theory implies that a Lipschitz function on $\Ab(M,A)$ has values exponentially concentrated around the expectation. Then in Section \ref{sec:FUPC} we apply such theory to the above exponential sum as a function on $\Ab(M,A)$. Indeed, we show that this exponential sum (when $k\ne j$) always obeys a ``square-root cancellation'' so it is bounded by $M^{2\ve}\sqrt A$ for any $\ve>0$ (except a set of alphabets of exponentially small probability). Putting this estimate back to the one for $r_1$, we then have that $r_1^2\le_\ve M^{2\ve}\sqrt A\cdot A/M$ so 
$$r_1\le_\ve M^{-\left(\frac12-\frac34\delta+\ve\right)},$$
noticing that $A=M^\delta$ by \eqref{eq:delta}. Hence, the FUP in Theorem \ref{thm:FUPC} follows.

\section{Concentration of measure in the space of alphabets}\label{sec:com}
The concentration of measure theory in part is concerned with the phenomenon that in some metric spaces of large dimension and equipped with certain probability measure, any Lipschitz function has its values exponentially concentrated around the expectation. See the monographs Milman-Schechtman \cite{MS} and Ledoux \cite{L}. 

In this section, we establish such a theory in the space of alphabets $\Ab(M,A)$. Equip $\Ab(M,A)$ with the uniform counting probability measure $\mu$ \eqref{eq:prob}. Set the metric in $\Ab(M,A)$ by
\begin{equation}\label{eq:dist}
d(\Ac_1,\Ac_2)=\Card\left(\Ac_1\triangle\Ac_2\right)=\Card(\Ac_1\setminus\Ac_2)+\Card(\Ac_2\setminus\Ac_1)\quad\text{for }\Ac_1,\Ac_2\in\Ab(M,A).
\end{equation}
Here, $\Ac_1\triangle\Ac_2$ denotes the symmetric difference.  Now, for any function $F:\Ab(M,A)\to\C$, its Lipschitz norm is
$$\|F\|_\Lip=\max_{\Ac_1,\Ac_2\in\Ab(M,A),\Ac_1\ne\Ac_2}\frac{|F(\Ac_1)-F(\Ac_2)|}{d(\Ac_1,\Ac_2)}.$$
Under this setup, we have that
\begin{thm}[Concentration of measure in the space of alphabets]\label{thm:comA}
Let $M,A\in\N$ with $M\ge3$ and $1<A<M$. Then for any function $F:\Ab(M,A)\to\C$ and $t>0$,
$$\mu\Big(\left\{\Ac\in\Ab(M,A):|F(\Ac)-\Eb(F)|\ge t\right\}\Big)\le2\exp\left(-\frac{t^2}{16A\|F\|^2_\Lip}\right),$$
in which $\Eb(F)$ is the expectation of $F$ with respect to $\mu$.
\end{thm}

Theorem \ref{thm:comA} is largely inspired by Maurey \cite{Ma} and Schechtman \cite{Sche}, in which they studied the concentration of measure theory in the space of permutations. See also McDiarmid \cite{Mc} and Talagrand \cite{Tal} for a more general theory on symmetric groups. 

\begin{rmk}
A similar result to Theorem \ref{thm:comA} was obtained by Greenhill-Isaev-Kwan-McKay \cite[Section 2.2]{GIKM} using a different argument. We thank Jie Ma for this reference.
\end{rmk}

In Subsection \ref{sec:comM}, we quote a relevant concentration of measure theorem for finite metric spaces as presented in Milman-Schechtman \cite[Section 7]{MS} and Ledoux \cite[Section 4.1]{L}. In Subsection \ref{sec:comP}, we apply such a theorem to the space of permutations and then in Subsection \ref{sec:comA}, we deduce our own based on the one for the permutations.

\subsection{Finite metric spaces}\label{sec:comM}
We first need a concept of ``length'' of finite metric spaces, as described in Milman-Schechtman \cite[Section 7.7]{MS}.
\begin{defn}[Lengths of finite metric spaces]\label{defn:length}
Let $(X,d)$ be a finite metric space. We say that $(X,d)$ is of length at most $l$ if there exist positive numbers $a_1,...,a_n$ with 
\begin{equation}\label{eq:length}
l=\left(\sum_{k=1}^n|a_k|^2\right)^{\frac12}
\end{equation}
and a sequence $\{\Xc^k\}_{k=0}^n$, $\Xc^k=\{\Omega^k_j\}_{j=1}^{m_k}$, of partitions of $X$ with the following properties.
\begin{enumerate}[(i).]
\item $m_0=1$, i.e., $\Xc^0=\{X\}$,
\item $m_n=\Card(X)$, i.e., $\Xc^n=\{\{x\}:x\in X\}$,
\item $\Xc^k$ is a refinement of $\Xc^{k-1}$ for all $k=1,...,n$,
\item for all $k=1,...,n$, $r=1,...,m_{k-1}$ and $p,q$ such that $\Omega^k_p,\Omega^k_q\subset\Omega^{k-1}_r$, there exists a bijection $\phi:\Omega^k_p\to\Omega^k_q$ with $d(x,\phi(x))\le a_k$ for all $x\in\Omega^k_p$.
\end{enumerate}
\end{defn}

\begin{rmk}
Taking the trivial sequence of partitions $\Xc^0=\{X\}$ and $\Xc^1=\{\{x\}:x\in X\}$, we see that the above conditions hold with 
$$l=\diam(X):=\max_{x,y\in X}d(x,y).$$
It then follows that the length of a finite metric space is at most $\diam(X)$.
\end{rmk}

Let $F:X\to\C$. We remind the reader that the Lipschitz norm of $F$ is
$$\|F\|_\Lip=\max_{x_1,x_2\in X,x_1\ne x_2}\frac{|F(x_1)-F(x_2)|}{d(x_1,x_2)}.$$
Under these notations, as taken from Milman-Schechtman \cite[Section 7.8]{MS}, the concentration of measure phenomenon states that 
\begin{thm}[Concentration of measure in finite metric spaces ]\label{thm:com}
Let $(X,d)$ be a finite metric space of length at most $l$. Suppose that $\mu$ is the uniform counting probability measure on $X$. Then for any function $F:X\to\C$ and $t>0$,
$$\mu\Big(\{x\in X:|F(x)-\Eb(F)|\ge t\}\Big)\le2\exp\left(-\frac{t^2}{4l^2\|F\|^2_\Lip}\right),$$
in which $\Eb(F)$ is the expectation of $F$ with respect to $\mu$.
\end{thm}

\subsection{From metric spaces to spaces of permutations}\label{sec:comP}
Let $\Pi(M,A)$ be the space of permutations of $A$ elements from $\{0,...,M-1\}$, that is, each $\pi\in\Pi(M,A)$ is an injective mapping $\pi:\{0,...,A-1\}\to\{0,...,M-1\}$. Equip $\Pi(M,A)$ with the metric
$$d^p(\pi_1,\pi_2)=\Card\{j=0,...,A-1:\pi_1(j)\ne\pi_2(j)\}\quad\text{for }\pi_1,\pi_2\in\Pi(M,A).$$
Here and thereafter, we use the superscript $p$ to indicate the objects associated with permutations. Equip $\Pi(M,A)$ with the uniform counting probability measure $\mu^p$, i.e.,
$$\mu^p(\Omega)=\frac{\Card(\Omega)}{\Card(\Pi(M,A))}=\frac{(M-A)!}{M!}\cdot\Card(\Omega)\quad\text{for any }\Omega\subset\Pi(M,A).$$
Then we have that
\begin{thm}[Concentration of measure in the space of permutations]\label{thm:comP}
Suppose that $M,A\in\N$ with $M\ge3$ and $1<A<M$. Then for any function $F:\Pi(M,A)\to\C$ and $t>0$,
$$\mu^p\Big(\{\pi\in\Pi(M,A):|F(\pi)-\Eb^p(F)|\ge t\}\Big)\le2\exp\left(-\frac{t^2}{16A\|F\|^2_\Lip}\right),$$
in which $\Eb^p(F)$ is the expectation of $F$ with respect to $\mu^p$.
\end{thm}
\begin{proof}[Proof of Theorem \ref{thm:comP}]
To apply the concentration of measure theory in finite metric spaces, Theorem \ref{thm:com}, we need to estimate the length of $\Pi(M,A)$. To this end, we construct a sequence $\{\Xc^k\}_{k=0}^A$ of partitions as follows. For $k=0$, assign $\Xc^0=\Pi(M,A)$. For $k=1,...,A$, write
$$\Xc^k=\left\{\Omega_{j_0\cdots j_{k-1}}:j_0,...,j_{k-1}\text{ are distinct in }\{0,...,M-1\}\right\},$$
in which
$$\Omega_{j_0\cdots j_{k-1}}=\{\pi\in\Pi(M,A):\pi(0)=j_0,...,\pi(k-1)=j_{k-1}\}.$$
That is, $\Omega_{j_0\cdots j_{k-1}}$ is the collection of permutations such that the first $k$ elements are mapped to $j_0,...,j_{k-1}$. 

In Definition \ref{defn:length}, Conditions (i), (ii), and (iii) are clearly valid. For Condition (iv), choose 
$\Omega_{j_0\cdots j_{k-1}r},\Omega_{j_0\cdots j_{k-1}s}\subset\Omega_{j_0\cdots j_{k-1}}$. Let $\tau$ be the transposition that switches $r$ with $s$. Then $\pi(\{0,...,A-1\})$ and $\tau\circ\pi(\{0,...,A-1\})$ differ by exactly two elements. Define $\phi=\tau\circ\pi:\Omega_{j_0\cdots j_{k-1}r}\to\Omega_{j_0\cdots j_{k-1}s}$. It follows that
$$d^p(\pi,\phi(\pi))\le2.$$
Thus, we can take $a_k=2$ for all $k=1,...,A$. The length of $\Pi(M,A)$ is then computed by \eqref{eq:length} and is at most
$$\left(\sum_{k=1}^A|a_k|^2\right)^\frac12=2A^\frac12.$$
The proof is complete after applying Theorem \ref{thm:com}.
\end{proof}

\subsection{From permutations to alphabets}\label{sec:comA}
In this section, we establish our relevant concentration of measure theorem in the space of alphabets $\Ab(M,A)$ (Theorem \ref{thm:comA}) through a slight detour. That is, we build upon the corresponding theorem in the space of permutations $\Pi(M,A)$, which in turn depends on the length estimate of $\Pi(M,A)$ (Definition \ref{defn:length}). 

We adapt the concentration of measure result for permutations given in Theorem \ref{thm:comP} to that for alphabets in Theorem \ref{thm:comP}. 

Recall that any permutation $\pi\in\Pi(M,A)$ is an injective mapping from $\{0,...,A-1\}$ to $\{0,...,M-1\}$. Define $P:\Pi(M,A)\to\Ab(M,A)$ by $P(\pi)=\mathrm{Im}(\pi)$, the image of $\pi$, for $\pi\in\Pi(M,A)$. (That is, the mapping $P$ removes the order in the permutation.) Then $P$ is surjective and $\Card(P^{-1}(\Ac))=A!$ for any $\Ac\in\Ab(M,A)$. 

Let $F:\Ab(M,A)\to\C$. Then $F$ naturally induces a function $F^p:\Pi(M,A)\to\C$ by $F^p=F\circ P$. For the concentration of measure theory, we need to compare the expectation (with respect to different probability measures $\mu$ and $\mu^p$) and Lipschitz norms (with respect to different metrics $d$ and $d^p$) of $F$ and $F^p$:

$\bullet$ Expectation:
\begin{eqnarray*}
\Eb^p(F^p)&=&\frac{1}{\Card(\Pi(M,A))}\sum_{\pi\in\Pi(M,A)}F^p(\pi)\\
&=&\frac{(M-A)!}{M!}\sum_{\pi\in\Pi(M,A)}F\circ P(\pi)\\
&=&\frac{(M-A)!}{M!}\sum_{\Ac\in\Ab(M,A)}\sum_{P(\pi)=\Ac}F(\Ac)\\
&=&\frac{(M-A)!}{M!}\sum_{\Ac\in\Ab(M,A)}A!\cdot F(\Ac)\\
&=&\frac{1}{\Card(\Ab(M,A))}\sum_{\Ac\in\Ab(M,A)}F(\Ac)\\
&=&\Eb(F).
\end{eqnarray*}

$\bullet$ Lipschitz norm: Notice that if $P(\pi_1)=P(\pi_2)$, then 
$$F^p(\pi_1)=F(P(\pi_1))=F(P(\pi_2))=F^p(\pi_2).$$
Hence, it suffices to consider the case when $P(\pi_1)\ne P(\pi_2)$. In this case,
$$d^p(\pi_1,\pi_2)\ge d(P(\pi_1),P(\pi_2)).$$
It is based on a simple observation that if the combinations $P(\pi_1)$ and $P(\pi_2)$ satisfy that 
$$d(P(\pi_1),P(\pi_2))=\Card(P(\pi_1)\triangle P(\pi_2))=k,$$
then 
$$d^p(\pi_1,\pi_2)=\Card\{j=0,...,A-1:\pi_1(j)\ne\pi_2(j)\}\}\ge k.$$
Therefore,
\begin{eqnarray*}
\|F^p\|_\Lip&=&\max_{\pi_1,\pi_2\in\Pi(M,A),P(\pi_1)\ne P(\pi_2)}\frac{|F^p(\pi_1)-F^p(\pi_2)|}{d^p(\pi_1,\pi_2)}\\
&\le&\max_{\pi_1,\pi_2\in\Pi(M,A),P(\pi_1)\ne P(\pi_2)}\frac{|F(P(\pi_1))-F(P(\pi_2))|}{d(P(\pi_1),P(\pi_2))}\\
&=&\max_{\Ac_1,\Ac_2\in\Ab(M,A),\Ac_1\ne\Ac_2}\frac{|F(\Ac_1)-F(\Ac_2)|}{d(\Ac_1,\Ac_2)}\\
&=&\|F\|_\Lip.
\end{eqnarray*}
With these facts, we now prove Theorem \ref{thm:comA}:
\begin{proof}[Proof of Theorem \ref{thm:comA}]
Let $F:\Ab(M,A)\to\C$ and $t>0$. With $F^p=F\circ P:\Pi(M,A)\to\C$, we have that $\Eb^p(F^p)=\Eb(F)$ as shown above. Thus,   
$$\left|F^p(\pi)-\Eb^p(F^p)\right|\ge t$$
if and only if $P(\pi)=\Ac$ for some $\Ac\in\Ab(M,A)$ such that
$$\left|F(\Ac)-\Eb(F)\right|\ge t.$$
Hence,
\begin{eqnarray*}
&&\mu^p\Big(\{\pi\in\Pi(M,A):|F^p(\pi)-\Eb^p(F^p)|\ge t\}\Big)\\
&=&\frac{1}{\Card(\Pi(M,A))}\cdot\Card\left(\{\pi\in\Pi(M,A):|F^p(\pi)-\Eb^p(F^p)|\ge t\}\right)\\
&=&\frac{(M-A)!}{M!}\cdot\Card\left(\{\pi\in\Pi(M,A):|F(P(\pi))-\Eb(F)|\ge t\}\right)\\
&=&\frac{(M-A)!}{M!}\cdot A!\cdot\Card\left(\{\Ac\in\Ab(M,A):|F(\Ac)-\Eb(F)|\ge t\}\right)\\
&=&\frac{1}{\Card(\Ab(M,A))}\cdot\Card\left(\{\Ac\in\Ab(M,A):|F(\Ac)-\Eb(F)|\ge t\}\right)\\
&=&\mu\Big(\{\Ac\in\Ab(M,A):|F(\Ac)-\Eb(F)|\ge t\}\Big).
\end{eqnarray*}
By Theorem \ref{thm:comP} and that $\|F^p\|_\Lip\le\|F\|_\Lip$ as shown above,
\begin{eqnarray*}
&&\mu\Big(\{\Ac\in\Ab(M,A):|F(\Ac)-\Eb(F)|\ge t\}\Big)\\
&=&\mu^p\Big(\{\pi\in\Pi(M,A):|F^p(\pi)-\Eb^p(F^p)|\ge t\}\Big)\\
&\le&2\exp\left(-\frac{t^2}{16A\|F^p\|^2_\Lip}\right)\\
&\le&2\exp\left(-\frac{t^2}{16A\|F\|^2_\Lip}\right),
\end{eqnarray*}
which is Theorem \ref{thm:comA}.
\end{proof}

\begin{rmk}
It would be interesting to derive an estimate of the length of $\Ab(M,A)$, which is expected to be at most $O(\sqrt A)$, as indicated in Theorem \ref{thm:com}. Such estimate leads to a direct proof of Theorem \ref{thm:comA} without using permutations. 
\end{rmk}

\section{Proof of Theorem \ref{thm:FUPC}}\label{sec:FUPC}
In this section, we prove the FUP for discrete Cantor sets with random alphabets in Theorem \ref{thm:FUPC}. Recall the notations: Let $M,A\in\N$ with $M\ge3$ and $1<A<M$ so $\delta=\log A/\log M\in(0,1)$. Equip the space of alphabets $\Ab(M,A)=\{\Ac\subset\{0,...,M-1\}:\Card(\Ac)=A\}$ with the uniform counting probability measure $\mu$ as in \eqref{eq:prob}. 

By the FUP for discrete Cantor sets $\Cc_k=\Cc_k(M,\Ac)$ as in \eqref{eq:Cantor} with random alphabets, we estimate
$$r_k=\|\Char_{\Cc_k}\Fc_N\Char_{\Cc_k}\|_{l^2_N\to l^2_N},$$
in which the alphabet $\Ac$ is chosen randomly from $\Ab(M,A)$ with respect to $\mu$. The starting point of the above estimate is the crucial submultiplicativity property proved by Dyatlov-Jin \cite[Section 3.1]{DJ1}: 
\begin{prop}\label{prop:subm}
For all $k_1,k_2\in\N$,
\begin{equation}\label{eq:sub}
r_{k_1+k_2}\le r_{k_1}r_{k_2}.
\end{equation}
\end{prop}
Hence, $r_k\le r_1^k$ and it suffices to establish the bound of $r_1=r_1(\Ac)$. In this case, $\Cc_1(M,\Ac)=\Ac$ and the mapping $\Char_\Ac\Fc_M\Char_\Ac:l^2_M\to l^2_M$ is given by a symmetric $M\times M$ matrix. Denote this matrix by $\Mcc(\Ac)$ and its rows by $R_j(\Ac)$, $j=0,...,M-1$. 

Let $\Mcc^\star$ be the complex conjugate of $\Mcc$. We estimate the largest singular value $r_1(\Ac)$ of $\Mcc(\Ac)$ via the analysis of $\Mcc(\Ac)\Mcc(\Ac)^\star$, which has entries 
\begin{eqnarray*}
F_{jk}(\Ac)&:=&R_j(\Ac)\ol{R_k(\Ac)}\\
&=&\frac{\Char_\Ac(j)\Char_\Ac(k)}{M}\sum_{l\in A}e^{\frac{2\pi i(k-j)l}{M}}\\
&=&\begin{cases}
\frac{A}{M}, & \text{if }j=k\in\Ac,\\
\frac{1}{M}\sum_{l\in\Ac}e^{\frac{2\pi i(k-j)l}{M}}, & \text{if }j,k\in\Ac\text{ and }j\ne k,\\
0, & \text{if }j\not\in\Ac\text{ or }k\not\in\Ac,
\end{cases}
\end{eqnarray*} 
in the view of the discrete Fourier transform \eqref{eq:FN}. For the exponential sum appeared above, we have that
\begin{prop}\label{prop:exp}
Suppose that $L>0$ and $m\in\Z\setminus\{0\}$. Then there exists $\Gb_{L,m}\subset\Ab(M,A)$ with
$$\mu\big(\Ab(M,A)\setminus\Gb_{L,m}\big)\le2e^{-\frac{L^2}{16}}$$
such that for all $\Ac\in\Gb_{L,m}$, 
$$\left|\sum_{l\in\Ac}e^{\frac{2\pi iml}{M}}\right|\le L\sqrt A.$$
\end{prop}

That is, outside a set of alphabets comprising exponentially small measure, the exponential sums exhibits ``square-root cancellation''.
\begin{proof}[Proof of Proposition \ref{prop:exp}]
Define the function $F:\Ab(M,A)\to\C$ by
$$F(\Ac)=\sum_{j\in\Ac}e^{\frac{2\pi imj}{M}}\quad\text{for }\Ac\in\Ab(M,A).$$
The expectation $\Eb(F)$ is straightforward:
\begin{eqnarray*}
\Eb(F)&=&\frac{1}{\Card(\Ab(M,A))}\sum_{\Ac\in\Ab(M,A)}\sum_{j\in\Ac}e^{\frac{2\pi imj}{M}}\\
&=&{M\choose A}^{-1}{M-1\choose A-1}\sum_{j=0}^{M-1}e^{\frac{2\pi imj}{M}}\\
&=&0.
\end{eqnarray*}
Here, we first interchange the sums and use the fact that for any fixed $j=0,...,M-1$, the number of combinations $\Ac$ which contains $j$ is ${M-1\choose A-1}$. Observe that for $\Ac_1,\Ac_2\in\Ab(M,A)$, 
$$\left|F(\Ac_1)-F(\Ac_2)\right|=\left|\sum_{j\in\Ac_1\setminus\Ac_2}e^{\frac{2\pi imj}{M}}-\sum_{j\in\Ac_2\setminus\Ac_1}e^{\frac{2\pi imj}{M}}\right|\le\Card(\Ac_1\triangle\Ac_2),$$
in the view of the metric \eqref{eq:dist}. Hence, $\|F\|_\Lip\le2$ and Theorem \ref{thm:comA} implies that
$$\mu\Big(\left\{\Ac\in\Ab(M,A):|F(\Ac)|\ge t\right\}\Big)\le2e^{-\frac{t^2}{16A}},$$
for any $t>0$. Taking $t=L\sqrt A$, we have the corollary.
\end{proof}

With these preparation, we prove Theorem \ref{thm:FUPC}:
\begin{proof}[Proof of Theorem \ref{thm:FUPC}]
Let $L>2$ be chosen later. By Proposition \ref{prop:exp}, for each $m\in\{-M+1,...,-1,1,...,M-1\}$, there exists $\Gb_{L,m}\subset\Ab(M,A)$ with
$$\mu\big(\Ab(M,A)\setminus\Gb_{L,m}\big)\le2e^{-\frac{L^2}{64}}$$
such that for all $\Ac\in\Gb_{L,m}$, 
$$\left|\sum_{l\in\Ac}e^{\frac{2\pi iml}{M}}\right|\le L\sqrt A.$$
Write
$$\Gb_{L,M}=\bigcap_{m=-M+1,...,-1,1,...,M-1}\Gb_{L,m}.$$
Then
\begin{equation}\label{eq:GL}
\mu\big(\Ab(M,A)\setminus\Gb_L\big)\le\sum_{m=-M+1,...,-1,1,...,M-1}\mu\big(\Ab(M,A)\setminus\Gb_{L,m}\big)\le4Me^{-\frac{L^2}{64}},
\end{equation}
moreover, for all $\Ac\in\Gb_L$ and all $m=-M+1,...,-1,1,...,M-1$, 
$$\left|\sum_{l\in\Ac}e^{\frac{2\pi iml}{M}}\right|\le L\sqrt A.$$
Pick any $\Ac\in\Gb_L$ and let $j\in\Ac$. Notice that $\{k-j:k=0,...,M-1\text{ and }k\ne j\}\subset\{-M+1,...,-1,1,...,M-1\}$ and contains $A-1$ numbers. Hence,
\begin{eqnarray*}
\sum_{k\in\Ac,k\ne j}\left|F_{jk}(\Ac)\right|&=&\frac{1}{M}\sum_{k\in\Ac,k\ne j}\left|\sum_{l\in\Ac}e^{\frac{2\pi i(k-j)l}{M}}\right|\le\frac{(A-1)L\sqrt A}{M}\le\frac{LA^\frac32}{M}.
\end{eqnarray*}
Since the diagonal terms $|F_{jj}|=A/M$, Schur's lemma (\cite[Lemma 3.8]{DJ1} and also \cite[Equation (4.11)]{Dy}) implies an upper bound of the eigenvalues of $\Mcc(\Ac)\Mcc(\Ac)^\star$:
$$\max_{j\in\Ac}\left\{\left|F_{jj}\right|+\sum_{k\in\Ac,k\ne j}\left|F_{jk}(\Ac)\right|\right\}\le\frac AM+\frac{LA^\frac32}{M}\le\frac{2LA^\frac32}{M}.$$
Since $A=M^\delta$, the largest eigenvalue of $\Mcc(\Ac)$ satisfies that
$$r_1(\Ac)\le\sqrt{\frac{2LA^\frac32}{M}}=\sqrt{2L}\cdot M^{-\left(\frac12-\frac34\delta\right)},$$
if $\delta<2/3$. For $\ve>0$, set $\sqrt{2L}=M^\ve$. Then $L=M^{2\ve}/2$ and
\begin{eqnarray*}
\beta^s(M,\Ac)&=&-\frac{\log r_1(\Ac)}{\log M}\\
&=&-\frac{\log\left[M^\ve\cdot M^{-\left(\frac12-\frac34\delta\right)}\right]}{\log M}\\
&\ge&\frac12-\frac34\delta-\ve,
\end{eqnarray*}
for all $\Ac\in\Gb_L$ (depending on $M$ and $\ve$) with
$$\mu\big(\Ab(M,A)\setminus\Gb_L\big)\le4Me^{-\frac{M^{4\ve}}{64}},$$
in the view of \eqref{eq:GL}.
\end{proof}

\section{Future investigations}
For $M,A\in\N$ with $M\ge3$ and $1<A<M$, still use $\Ab(M,A)$ as the space of alphabets with cardinality $A$. In this paper, we consider the discrete Cantor sets with \textit{random alphabets}, i.e.,
\begin{enumerate}[(i).]
\item[(i).] $$\Cc_k(M,\Ac)=\left\{\sum_{j=0}^{k-1}a_jM^j:a_0,...,a_{k-1}\in\Ac\right\}.$$
\end{enumerate}
In this randomization process, one chooses an alphabet $\Ac$ randomly from $\Ab(M,A)$ (with respect to the uniform counting probability measure $\mu$); then the Cantor sets are constructed in each step $j\in\N$ using the same alphabet $\Ac$ of digits. The resulting Cantor sets $\Cc_\infty$ \eqref{eq:Cinfty} always have dimension $\delta=\log A/\log M$. 

The concentration of measure phenomenon in Theorem \ref{thm:comA} emerges as $M\to\infty$ (and $A=M^\delta\to\infty$ as well when $\delta>0$). It is then responsible for the estimates of the exponent in the FUP for $\Cc_k(M,\Ac)$ with random alphabets (Theorems \ref{thm:FUPC} and \ref{thm:Eb}).

The randomization process in (i) is designed to rigorously verify the observation in Dyatlov-Jin \cite{DJ1} that the FUP with better exponent holds on average than the worst case, with respect to different alphabets of same cardinality. It is different from and should be compared with the \textit{random Cantor sets}, which have been extensively studied in the literature, see Falconer \cite[Section 15.1]{F}. For example, consider the following random ensemble.
\begin{enumerate}[(i).]
\item[(ii).] $$\Cc_k(M,\Ac_0,...,\Ac_{k-1})=\left\{\sum_{j=0}^{k-1}a_jM^j:a_0\in\Ac_0,a_1\in\Ac_1,...,a_{k-1}\in\Ac_{k-1}\right\},$$
in which $\Ac_j\in\Ab(M,A)$ for $j=0,...,k-1$ are identical and independent random variables. That is, the Cantor sets are constructed in each step $j\in\N$ using the random alphabet $\Ac_j$ of digits. 
\end{enumerate}
There are also other random models, see, for instance, Eswarathasan-Pramanik \cite[Section 8]{EP}. 

The randomization process (ii) differs with (i) significantly: Firstly, the random Cantor sets do not correspond to open quantum maps in Theorems \ref{thm:gapDJ} and Corollary \ref{cor:gap}, rather, the FUP associated with random Cantor sets are closely related to the random matrix theory\footnote{In fact, one can also interpret the $l^2_M\to l^2_M$ mapping norm estimate of $\Char_\Ac\Fc_M\Char_\Ac$ treated in Section \ref{sec:FUPC} from the random matrix theory point of view. However, the corresponding matrices are different from the typical random ensembles and are not discussed in the field to the authors' knowledge.}, particularly, the norm estimate of large random matrices, see Tao \cite[Section 2.3]{Tao}; secondly, the concentration of measure phenomenon in the space of random Cantor sets arises for fixed $M,A$ and as $k\to\infty$.

We shall leave the FUP in various random settings and applications for future investigations.

\section*{Acknowledgments}
We would like to thank Jie Ma for our discussions on the combinatorics involved in this paper as well as Semyon Dyatlov and Long Jin for their  comments and permission to use their numerical data from \cite{DJ1}. SE was supported by the NSERC Discovery Grant program during the writing of this article.


\begin{thebibliography}{99}

\bibitem[BD1]{BD1} J. Bourgain and S. Dyatlov, 
\textit{Fourier dimension and spectral gaps for hyperbolic surfaces}. Geom. Funct. Anal. 27 (2017), no. 4, 744--771. 

\bibitem[BD2]{BD2} J. Bourgain and S. Dyatlov, 
\textit{Spectral gaps without the pressure condition}. Ann. of Math. (2) 187 (2018), no. 3, 825--867.

\bibitem[CT]{CT} L. Cladek and T. Tao,
\textit{Additive energy of regular measures in one and higher dimensions, and the fractal uncertainty principle}. \href{https://arxiv.org/abs/2012.02747}{arXiv:2012.02747}.

\bibitem[Da]{Da} N. Dang,
\textit{Le principe d’incertitude fractal et ses applications}, [d'apr\`es Bourgain, Dyatlov, Jin, Nonnenmacher, Zahl]. S\'eminaire Bourbaki. Vol. 2020-2021. Expos\'es 1177.

\bibitem[Dy]{Dy} S. Dyatlov,
\textit{An introduction to fractal uncertainty principle}. J. Math. Phys. 60 (2019), no. 8, 081505, 31 pp.

\bibitem[DJ1]{DJ1} S. Dyatloy and L. Jin, 
\textit{Resonances for open quantum maps and a fractal uncertainty principle}. Comm. Math. Phys. 354 (2017), no. 1, 269--316.

\bibitem[DJ2]{DJ2} S. Dyatloy and L. Jin,
\textit{Semiclassical measures on hyperbolic surfaces have full support}. Acta Math. 220 (2018), no. 2, 297--339.

\bibitem[DJ3]{DJ3} S. Dyatloy and L. Jin, 
\textit{Dolgopyat's method and the fractal uncertainty principle}. Anal. PDE 11 (2018), no. 6, 1457--1485.

\bibitem[DJN]{DJN} S. Dyatlov, L. Jin, and S. Nonnenmacher,
\textit{Control of eigenfunctions on surfaces of variable curvature}. \href{https://arxiv.org/abs/1906.08923}{arXiv:1906.08923}. To appear in J. Amer. Math. Soc.

\bibitem[DyZa]{DyZa} S. Dyatlov and J. Zahl,
\textit{Spectral gaps, additive energy, and a fractal uncertainty principle}. Geom. Funct. Anal. 26 (2016), no. 4, 1011--1094.

\bibitem[DyZw]{DyZw} S. Dyatlov and M. Zworski, 
\textit{Fractal uncertainty for transfer operators}. Int. Math. Res. Not. IMRN 2020, no. 3, 781--812.

\bibitem[EP]{EP} S. Eswarathasan and M. Pramanik,
\textit{Restriction of Laplace-Beltrami eigenfunctions to arbitrary sets on manifolds}. \href{https://arxiv.org/abs/1901.07018}{arXiv:1901.07018}. To appear in Int. Math. Res. Not.

\bibitem[F]{F} K. Falconer, 
\textit{Fractal geometry}. Third edition. John Wiley \& Sons, Ltd., Chichester, 2014.

\bibitem[GIKM]{GIKM} C. Greenhill, M. Isaev, M. Kwan, and B. McKay, 
\textit{The average number of spanning trees in sparse graphs with given degrees}. European J. Combin. 63 (2017), 6--25.

\bibitem[GZ]{GZ} J. Galkowski and S. Zelditch
\textit{Lower bounds for Cauchy data on curves in a negatively curved surface}. \href{https://arxiv.org/abs/2002.09456}{arXiv:2002.09456}.

\bibitem[HS]{HS} R. Han and W. Schlag, 
\textit{A higher-dimensional Bourgain-Dyatlov fractal uncertainty principle}. Anal. PDE 13 (2020), no. 3, 813--863.

\bibitem[J1]{J1} L. Jin, 
\textit{Damped wave equations on compact hyperbolic surfaces}. Comm. Math. Phys. 373 (2020), no. 3, 771--794.

\bibitem[J2]{J2} L. Jin, 
\textit{Control for Schr\"odinger equation on hyperbolic surfaces}. Math. Res. Lett. 25 (2018), no. 6, 1865--1877.

\bibitem[JZ]{JZ} L. Jin and R. Zhang, 
\textit{Fractal uncertainty principle with explicit exponent}. Math. Ann. 376 (2020), no. 3-4, 1031--1057.

\bibitem[L]{L} M. Ledoux,
\textit{The concentration of measure phenomenon}. American Mathematical Society, Providence, RI, 2001.

\bibitem[Ma]{Ma} B. Maurey,
\textit{Construction de suites sym\'etriques}. C. R. Acad. Sci. Paris S\'er. A-B 288 (1979), no. 14, A679--A681.

\bibitem[Mc]{Mc} C. McDiarmid, 
\textit{On the method of bounded differences}. Surveys in combinatorics, 1989 (Norwich, 1989), 148--188,
London Math. Soc. Lecture Note Ser., 141, Cambridge Univ. Press, Cambridge, 1989.


\bibitem[MS]{MS} V. Milman and G. Schechtman,
\textit{Asymptotic theory of finite-dimensional normed spaces}. With an appendix by M. Gromov. Lecture Notes in Mathematics, 1200. Springer-Verlag, Berlin, 1986.

\bibitem[Sche]{Sche} G. Schechtman, 
\textit{L\'evy type inequality for a class of finite metric spaces}. Martingale theory in harmonic analysis and Banach spaces (Cleveland, Ohio, 1981), pp. 211--215, Lecture Notes in Math., 939, Springer, Berlin-New York, 1982.

\bibitem[Schw]{Schw} N. Schwartz,
\textit{The full delocalization of eigenstates for the quantized cat map}. \href{https://arxiv.org/abs/2103.06633}{arXiv:2103.06633}.

\bibitem[Tal]{Tal} M. Talagrand,
\textit{Concentration of measure and isoperimetric inequalities in product spaces}. Inst. Hautes \'Etudes Sci. Publ. Math. No. 81 (1995), 73--205.

\bibitem[Tao]{Tao} T. Tao, 
\textit{Topics in random matrix theory}. American Mathematical Society, Providence, RI, 2012.

\bibitem[W]{W} J. Wang, 
\textit{Strichartz estimates for convex co-compact hyperbolic surfaces}. Proc. Amer. Math. Soc. 147 (2019), no. 2, 873--883.

\end{thebibliography}
\end{document}